\documentclass[11pt]{article}

\usepackage[T1]{fontenc}
\usepackage[utf8]{inputenc}
\usepackage{lmodern}
\usepackage{amsmath,amssymb,amsthm,mathtools}
\usepackage[a4paper,margin=30mm]{geometry}
\usepackage{microtype}
\usepackage{enumitem}
\usepackage[colorlinks=true,linkcolor=blue,citecolor=blue,urlcolor=blue]{hyperref}

\newtheorem{theorem}{Theorem}[section]
\newtheorem{proposition}[theorem]{Proposition}
\newtheorem{corollary}[theorem]{Corollary}
\newtheorem{lemma}[theorem]{Lemma}
\theoremstyle{remark}
\newtheorem{remark}[theorem]{Remark}

\newcommand{\Will}{\mathcal W}
\newcommand{\R}{\mathbb R}
\newcommand{\Sph}{\mathbb S}
\newcommand{\dd}{\,\mathrm d}
\newcommand{\Area}{\operatorname{Area}}
\newcommand{\spt}{\operatorname{spt}}
\newcommand{\BL}{\mathrm{BL}}
\newcommand{\Var}{\mathbf V}
\newcommand{\TV}{\mathrm{TV}}

\numberwithin{equation}{section}

\title{Uniqueness of Finite-Time Varifold Limits for the
M\"obius-Invariant Willmore Flow}
\author{
Mohameden Ahmedou\thanks{Mathematisches Institut, Justus-Liebig-Universit\"at
Gie\ss en, Arndtstrasse 2, 35392 Gie\ss en, Germany.
E-mail: \texttt{Mohameden.Ahmedou@math.uni-giessen.de}.}
\and
Ruben Jakob\thanks{Mathematics Department, Technion--Israel Institute of
Technology, 3200003 Haifa, Israel.
E-mail: \texttt{rubenj@technion.ac.il}.}}
\date{November 8, 2026}

\hypersetup{
 pdfauthor={Mohameden Ahmedou and Ruben Jakob},
 pdftitle={Uniqueness of Finite-Time Varifold Limits for the
 M\"obius-Invariant Willmore Flow}
}

\begin{document}

\maketitle

\begin{abstract}
We prove uniqueness of finite-time geometric endpoints for the
M\"obius-invariant Willmore flow in $\Sph^3$ under uniform quantitative
nonumbilicity.  The multiplicity-counting varifolds converge, without
reparametrization or M\"obius renormalization, to a unique integral
two-varifold.  An intrinsic transport estimate gives quantitative
total-variation convergence of the induced area measures on the fixed
domain and bounded-Lipschitz Cauchy control of their pushforwards.
Together with Allard compactness and rectifiability, this upgrades
subsequential compactness to full-trajectory varifold convergence.  The
limit has generalized Euclidean mean curvature in $L^2$ with the
natural endpoint lower-semicontinuity bound.  For finite maximal
trajectories with initial energy at most $8\pi$, Jakob's subsequential
alternative becomes sequence independent: the limit is zero, or it
has unit density and embedded Lipschitz support of genus zero or one.
For Hopf-torus trajectories under the same energy bound,
nonumbilicity is automatic and the anchored constant-speed profiles
converge weakly in $W^{2,2}$ and strongly in $W^{1,2}$ and
$C^{1,\alpha}$ for every $\alpha<\frac12$.  At infinite time, the same
method yields a unique limit under an additional finite-dissipation-length
condition.  
\end{abstract}



\section{Introduction and main results}

\subsection{The Willmore functional and the classical flow}

A recurring issue in the analysis of geometric flows is the gap
between subsequential compactness and convergence of the complete
trajectory.  Energy and curvature bounds may provide geometric cluster
points as time approaches an endpoint, but they do not by themselves
exclude different time sequences from selecting different limits.  In
parametric surface flows this problem is compounded by the freedom of
reparametrization and by the possible noncompactness of the ambient
symmetry group.  The purpose of this paper is to identify a direct
measure-transport mechanism that removes this sequence dependence for
the M\"obius-invariant Willmore flow under quantitative
nonumbilicity.

For a closed immersed surface $F\colon\Sigma\to\Sph^3$, where
$\Sph^3\subset\R^4$ is the unit round three-sphere with metric
$g_{\Sph^3}$, the Willmore functional in the normalization used below is
\[
 \Will(F)=\int_\Sigma\left(1+\frac14|H|^2\right)\dd\mu.
\]
Here $H$ is the trace mean-curvature vector in $\Sph^3$ and
$\mu=\mu_{F^*g_{\Sph^3}}$ is the induced area measure.
Equivalently, after stereographic projection, this is the Euclidean
Willmore energy.  Its invariance under conformal transformations of the
ambient sphere is the source of both its geometric rigidity and its
noncompactness.  The Li--Yau inequality \cite{LiYau} shows, in
particular, that the threshold $8\pi$ excludes self-intersections for
closed immersions of smaller energy.  For tori, the Willmore conjecture
proved by Marques and Neves \cite{MarquesNeves} identifies $2\pi^2$ as
the optimal lower bound and characterizes equality by the conformal
Clifford-torus family.  These two thresholds play different roles: the
first is a compactness and multiplicity threshold, whereas the second
is the variational minimum among all closed immersed surfaces in
$\Sph^3$ of genus at least one.

The negative $L^2$-gradient flow of $\Will$ is the classical Willmore
flow.  It was introduced independently by Kuwert--Sch\"atzle
\cite{KuwertSchaetzleSmall} and Simonett \cite{Simonett} in 2001.
Kuwert and Sch\"atzle established the small-energy theory and developed
the general fourth-order evolution and its concentration analysis in
\cite{KuwertSchaetzleGradient}; later refinements include the sharp
small-energy asymptotics of Kuwert--Scheuer \cite{KuwertScheuer}.
The global picture is nevertheless subtle.  Embeddedness need not be
preserved \cite{MayerSimonett}, and finite-time singular examples are
known \cite{Blatt}.  Dall'Acqua, M\"uller, Sch\"atzle, and Spener
\cite{DallAcquaMuellerSchaetzleSpener} obtained a detailed analysis for
tori of revolution, including degenerating examples.  In the
parametric formulation of Palmurella and Rivi\`ere
\cite{PalmurellaRiviereApproach,PalmurellaRiviereFlow}, finite-time loss
of compactness for the classical flow is organized into energy
concentration, possibly producing a branch point, and degeneration of
the conformal class.  We refer to \cite{LanMartinoRiviere} for a broad
account of the classical flow and of the surrounding Willmore theory.

Compactness for sequences with bounded Willmore energy is also built
on the weak theory of conformal immersions.  The results of
Kuwert--Li \cite{KuwertLi}, Rivi\`ere
\cite{RiviereAnalysis,RiviereDegenerating}, and Sch\"atzle
\cite{SchaetzleConformalFactor} separate convergence of the geometric
image from possible changes of parametrization and conformal
structure.  At the level of geometric images, Allard's compactness
theorem \cite{Allard} provides the natural integral-varifold framework.
It is precisely at this parametrization-independent level that the
endpoint question considered here is most naturally formulated.

\subsection{The M\"obius-invariant Willmore flow}

The evolution studied here is not the classical Willmore flow.  The
M\"obius-invariant Willmore flow (MIWF) is
\[
 \partial_tF=-|A^0|^{-4}\nabla_{L^2}\Will(F),
\]
where $A^0$ denotes the trace-free second fundamental form.  The
weight $|A^0|^{-4}$ corrects the conformal scaling of the classical
gradient and makes the evolution covariant under fixed M\"obius
transformations without a rescaling of time.  The price is a singular
coefficient at umbilic points.  Since an umbilic-free compact oriented
surface has vanishing Euler characteristic, the natural closed domain
for the flow is a torus.

The analytic development of the MIWF began with Jakob's short-time
existence and uniqueness theorem \cite{JakobShortTime}; functional
analytic properties and regularity were established in
\cite{JakobRegularity}.  The Hopf fibration relates distinguished
toroidal trajectories to fourth-order curve flows on $\Sph^2$.  The
corresponding curve geometry and reduction mechanisms were developed
in Jakob's work on Willmore flow of Hopf tori \cite{JakobHopf}, building
on Pinkall's description of Hopf tori as inverse images of closed
curves under the Hopf map \cite{Pinkall}.  We use only the equivariance
and geometric reduction formulas that are unaffected by the later
corrections in \cite{JakobHopfCorrection}; in particular, we do not use
the corrected assertions in parts \textup{(II)}--\textup{(III)} of
Theorem~1 or in the second part of Proposition~6 of \cite{JakobHopf}.
Under global compactness hypotheses, Jakob
proved global existence, full convergence, and stability results for
the MIWF in \cite{JakobGlobal}.

The endpoint compactness theory is markedly less rigid.  Jakob's
subsequent theorem \cite[Theorem~1.1]{Jakob} treats toroidal MIWF
trajectories in $\Sph^3$ with initial energy at most $8\pi$.  From
every sequence of times approaching the maximal time it extracts a
subsequence whose multiplicity-one image measures converge weakly to
the weight of an integral rectifiable varifold.  The limit is zero, or
its support is a closed embedded orientable Lipschitz surface of genus
zero or one; in the genus-one case the support admits a uniformly
conformal bi-Lipschitz parametrization in
$W^{2,2}\cap W^{1,\infty}$.  For trajectories in the Hopf sector every
nonzero subsequential support is an embedded $C^1$ Hopf torus
\cite[Theorem~1.2]{Jakob}.  These results identify all possible
subsequential endpoints, but they do not show that two sequences of
times along the same trajectory have the same limit.

\subsection{Main result and relation to the preceding theory}

We close this gap by estimating the evolution of the geometric image
measure itself.  The resulting estimate is independent of the choice
of coordinates on the torus and does not require a time-dependent
M\"obius normalization.  Under a uniform quantitative nonumbilicity
hypothesis it makes the image measures Cauchy at every finite endpoint;
integral compactness then upgrades the measure limit to a unique
varifold limit.  Our principal result is the following.

\begin{theorem}[Unique finite-time limit]
\label{thm:main}
Let $F\colon[0,T)\times\Sigma\to\Sph^3$ be a smooth MIWF of a compact
torus, where $T<\infty$, and write $F_t:=F(t,\cdot)$.  Let $A_t^0$
denote the trace-free second fundamental form of $F_t$.  Put
\[
 W(t):=\Will(F_t),\qquad W_0:=W(0),\qquad
 \mu_t:=\mu_{F_t^*g_{\Sph^3}}.
\]
Assume that
for some $c_0>0$,
\[
 \inf_{0<t<T}\min_{x\in\Sigma}|A_t^0(x)|^2\geq c_0.
\]
Let $\Var_t$ denote the multiplicity-counting integral two-varifold in
$\R^4$ induced by $F_t$, and set
$\lambda_t:=\lVert\Var_t\rVert=(F_t)_\#\mu_t$.
Then there exists a unique integral two-varifold
$\Var_T$ in $\R^4$, possibly the zero varifold, such that, with
$\rightharpoonup$ denoting weak varifold convergence,
\begin{equation}
 \label{eq:main-convergence}
 \Var_t\rightharpoonup\Var_T
 \qquad\text{as }t\nearrow T.
\end{equation}
No reparametrizations of $\Sigma$ and no M\"obius transformations of
$\Sph^3$ are required.

More precisely, there are finite nonnegative Radon measures $\mu_T$ on
$\Sigma$ and $\lambda_T=\lVert\Var_T\rVert$ on $\Sph^3$ such that
$\mu_t\to\mu_T$ in total variation and
$\lambda_t\to\lambda_T$ in the bounded-Lipschitz distance
$d_{\BL}$ defined in \eqref{eq:bl-distance}.  Setting
\[
 W_T:=\lim_{t\nearrow T}W(t),
\]
there holds, for every $0<s<T$,
\begin{align}
 \label{eq:main-endpoint-tv}
 \lVert\mu_T-\mu_s\rVert_{\TV}
 &\leq \frac{2\sqrt{W_0}}{c_0}
 \sqrt{(T-s)\bigl(W(s)-W_T\bigr)},\\
 \label{eq:main-endpoint-bl}
 d_{\BL}(\lambda_T,\lambda_s)
 &\leq \frac{2\sqrt{W_0}}{c_0}
 \sqrt{(T-s)\bigl(W(s)-W_T\bigr)}.
\end{align}
Moreover, $\Var_T$ has generalized Euclidean mean-curvature vector
$H_{\Var_T}^{\R^4}\in L^2(\lVert\Var_T\rVert;\R^4)$ and
\begin{equation}
 \label{eq:main-curvature-lsc}
 \frac14\int_{\R^4}|H_{\Var_T}^{\R^4}|^2
 \dd\lVert\Var_T\rVert\leq W_T.
\end{equation}
\end{theorem}

When $T=T_{\max}(F_0)<\infty$, we call the resulting maximal
trajectory a \emph{singular flow line}, consistently with
\cite[Definition~2.1]{Jakob}.  The theorem itself is slightly more
general, since it applies on every finite interval on which the stated
uniform nonumbilicity bound holds.

The theorem concerns the original, unrenormalized trajectory.  Its
main novelty is not the existence of cluster varifolds but a mechanism
that forces all cluster varifolds to coincide.  The intrinsic
transport identity used below follows the pushforward
$(F_t)_\#\mu_{F_t^*g_{\Sph^3}}$ and therefore records simultaneously
the motion of the immersion and the variation of its area element.  It
yields the quantitative endpoint estimates
\eqref{eq:main-endpoint-tv}--\eqref{eq:main-endpoint-bl}, with a bound
expressed solely through the remaining energy dissipation.  Once the
weight measure is unique, integral compactness and rectifiability
determine the approximate tangent plane almost everywhere and hence
the limiting varifold.  This is why an estimate only for the area
measures on the fixed parameter domain would be insufficient.

The argument also retains geometric information at the endpoint.  The
uniform Willmore bound passes to an $L^2$ bound for the generalized
Euclidean mean curvature of the limit, giving
\eqref{eq:main-curvature-lsc}.  Combining
Theorem~\ref{thm:main} with \cite[Theorem~1.1]{Jakob}, the zero,
genus-zero, and genus-one alternatives below the $8\pi$ threshold
become alternatives for the complete trajectory rather than for an
extracted subsequence.  Every nonzero unrenormalized image converges
to the unique support in Hausdorff distance.  In the genus-one case
the limiting parametrizations constructed in \cite{Jakob} represent
the same multiplicity-one varifold, although no uniqueness of
parametrization is asserted.

The Hopf sector provides a geometric class in which the hypothesis of
Theorem~\ref{thm:main} is automatic: if $\kappa$ is the geodesic
curvature of the profile curve, then
$|A^0|^2=2(1+\kappa^2)\geq2$.  The unique nonzero finite-time support is
therefore, under the same $8\pi$ energy bound, an embedded $C^1$ Hopf
torus.  More is true after fixing the
orientation, base point, and constant-speed gauge.
Theorem~\ref{thm:anchored-profile} proves convergence of the entire
anchored profile trajectory, and
Corollary~\ref{cor:pinkall-modulus} identifies the resulting limit in
the moduli space through Pinkall's lattice description.

At infinite time, the finite-interval factor in the transport estimate
is no longer bounded.  Proposition~\ref{prop:infinite-finite-length}
isolates the additional condition needed by this method: uniqueness
holds if the square root of the energy dissipation is integrable in
time.  A \L ojasiewicz--Simon inequality would imply such a condition
only after a separate trapping and gauge theorem near a smooth critical
immersion; no such theorem is assumed here.

We finally delimit the finite-time conclusion.  For a general maximal
trajectory with $T_{\max}<\infty$, either quantitative
nonumbilicity fails along a sequence approaching $T_{\max}$, or
Theorem~\ref{thm:main} gives a unique zero, genus-zero, or genus-one
endpoint.  Thus the exhaustive alternative still contains four
branches: an umbilic event and the three geometric endpoints.  The
result is an endpoint theorem, not a continuation theorem.  In
particular, the available weak $W^{2,2}$ and weak-* $W^{1,\infty}$
convergence does not transmit the pointwise lower bound for $|A^0|^2$
to a toroidal limiting parametrization, whereas the existing
short-time theory starts from $W^{4-4/p,p}$ data with $p>3$.  Excluding
a toroidal endpoint or extending the flow beyond $T_{\max}$ would
therefore require additional regularity and well-posedness results.
Likewise, conclusions from the classical parametric Willmore flow
\cite{PalmurellaRiviereApproach,PalmurellaRiviereFlow} cannot be
transferred directly to the MIWF, whose principal weight and conformal
covariance are different.

\subsection{Organization of the paper}

Section~2 fixes the geometric conventions.  Section~3 proves the
intrinsic transport and Cauchy estimates.  Section~4 proves
Theorem~\ref{thm:main} and records a finite-dissipation-length
criterion at infinite time.  Section~5 derives the exhaustive finite-time
endpoint alternative and records the unresolved continuation problem.
Section~6 treats Hopf tori, proves the anchored-profile estimate, and
derives the unique Pinkall-modulus limit.

\section{Setting, conventions, and assumptions}

Let $\Sigma$ be a smooth compact torus and let
\[
 F\colon [0,T)\times\Sigma\longrightarrow \Sph^3\subset\R^4,
 \qquad F_t:=F(t,\cdot),
\]
be a smooth solution of the M\"obius-invariant Willmore flow on its
interval of existence.  The local well-posedness and regularity theory
for this evolution is developed in
\cite{JakobShortTime,JakobRegularity}.  Here $\Sph^3$ is the unit
sphere, $g_{\Sph^3}$ is its round metric, and $g_{\mathrm{euc}}$ is the
Euclidean metric of $\R^4$.  When the interval is maximal, we denote
its right endpoint by $T_{\max}(F_0)$.  We write
\[
 g_t:=F_t^*g_{\Sph^3}=F_t^*g_{\mathrm{euc}},
 \qquad \mu_t:=\mu_{g_t},
 \qquad \dd\mu_t:=\dd\mu_{g_t},
\]
and denote by $A_t$, $A_t^0$, and $H_t$ respectively the second
fundamental form, its trace-free part, and the trace mean-curvature
vector of $F_t$ in $\Sph^3$.  We set
\[
 \Area(F_t):=\mu_t(\Sigma).
\]
Unless an ambient space is displayed explicitly, norms and inner
products of tangent and normal tensors are taken with respect to
$g_t$ and $g_{\Sph^3}$.  We write $\nabla^{\Sph^3}$ for the spherical
gradient. Our normalization of the Willmore functional is
\begin{equation}
 \label{eq:willmore}
 \Will(F_t)
 :=\int_\Sigma
 \left(1+\frac14\lvert H_t\rvert^2\right)\dd\mu_t.
\end{equation}
If $H_t^{\R^4}$ denotes the Euclidean trace mean-curvature vector, then
\[
 H_t^{\R^4}=H_t-2F_t,
 \qquad
 \lvert H_t^{\R^4}\rvert^2
 =\lvert H_t\rvert^2+4,
\]
and consequently
\begin{equation}
 \label{eq:euclidean-willmore}
 \Will(F_t)=\frac14\int_\Sigma
 \lvert H_t^{\R^4}\rvert^2\dd\mu_t.
\end{equation}
Throughout the paper we write
\[
 W(t):=\Will(F_t),
 \qquad
 W_0:=W(0)=\Will(F_0).
\]

Let
\[
 G_t:=\nabla_{L^2}\Will(F_t)
\]
be the $L^2(\Sigma,\dd\mu_t)$-gradient of $\Will$, normalized by
\[
 D\Will(F_t)[\phi]
 =\int_\Sigma\langle G_t,\phi\rangle\dd\mu_t
\]
for every smooth normal variation field
$\phi\in\Gamma(N_{F_t}\Sigma)\subset\Gamma(F_t^*T\Sph^3)$.
Here $D\Will(F_t)$ is the first variation of $\Will$,
$N_{F_t}\Sigma$ is the normal bundle of $F_t(\Sigma)$ in $\Sph^3$,
and $\Gamma(E)$ denotes the space of smooth sections of a bundle $E$.
The MIWF equation is
\begin{equation*}
 \label{eq:miwf}
 \partial_tF_t
 =-\frac{1}{\lvert A_t^0\rvert^4}G_t.
 \tag{MIWF}
\end{equation*}
Equivalently, with
\[
 a_t:=\lvert A_t^0\rvert^2,
\]
one has $\partial_tF_t=-a_t^{-2}G_t$.  These conventions give the
dissipation identity
\begin{equation}
 \label{eq:dissipation}
 -W'(t)
 =\int_\Sigma a_t^{-2}\lvert G_t\rvert^2\dd\mu_t.
\end{equation}
Indeed, this follows directly by pairing \eqref{eq:miwf} with $G_t$.
Equivalently, if the spherical Willmore operator is written as
\[
 \mathcal L_{\Sph^3}(F_t)
 :=\Delta_{F_t}^{\perp,\Sph^3}H_t
   +Q(A_t^0)(H_t),
\]
then
\[
 G_t=\frac12\mathcal L_{\Sph^3}(F_t),
 \qquad
 \partial_tF_t=-\frac{1}{2a_t^2}\mathcal L_{\Sph^3}(F_t).
\]
Here
\[
 Q(A_t^0)(H):=g_t^{ik}g_t^{j\ell}
 \langle (A_t^0)_{ij},H\rangle(A_t^0)_{k\ell}.
\]
Here $\Delta_{F_t}^{\perp,\Sph^3}$ is the normal-connection Laplacian
along $F_t$, and $(g_t^{ij})$ is the inverse matrix of $g_t$ in local
coordinates.
Thus the gradient, flow, and dissipation identities use one fixed
normalization; compare the first-variation formula in \cite{Weiner}.

The additional assumption is the uniform nonumbilicity condition
\begin{equation*}
 \label{eq:un}
 \boxed{
 \inf_{0<t<T}\min_{x\in\Sigma}
 \lvert A_t^0(x)\rvert^2\geq c_0>0.}
 \tag{UN}
\end{equation*}
This is stronger than requiring every individual immersion $F_t$ to
be umbilic-free.  Notice also that \eqref{eq:un} is an analytic
condition and is not the same as the ``nondegenerate case'' in
\cite{Jakob}, which means that a nonzero limiting surface has genus
one.

Throughout, ``Radon measure'' has its standard meaning and is scalar and
nonnegative unless it is explicitly described as signed or vector-valued.
In particular, every Radon measure on the compact spaces $\Sigma$ and
$\Sph^3$ is finite.

Let $G_2(\R^4):=\R^4\times G(4,2)$ be the Grassmann bundle of
unoriented two-planes over $\R^4$, and let $C_c(X)$ denote the space of
continuous, compactly supported real-valued functions on $X$.
Writing $\dd F_t$ for the differential of $F_t$, let $\Var_t$ be the
integral two-varifold in $\R^4$ associated with the immersion $F_t$,
defined by
\begin{equation}
 \label{eq:varifold-definition}
 \Var_t(\Phi)
 :=\int_\Sigma
 \Phi\bigl(F_t(x),\dd F_t(T_x\Sigma)\bigr)\dd\mu_t(x),
 \qquad \Phi\in C_c\bigl(G_2(\R^4)\bigr).
\end{equation}
For a Borel map $f$ and a Radon measure $\eta$, define the push-forward
by $(f_\#\eta)(B):=\eta(f^{-1}(B))$ for every Borel set $B$.
Its weight measure is
\begin{equation}
 \label{eq:weight-definition}
 \lambda_t:=\lVert\Var_t\rVert=(F_t)_\#\mu_t.
\end{equation}
For a two-varifold $V$ in $\R^4$, its first variation is the linear
functional
\begin{equation}
 \label{eq:first-variation-definition}
 \delta V(X)
 :=\int_{G_2(\R^4)}\operatorname{div}_S X(x)\dd V(x,S),
 \qquad X\in C_c^1(\R^4;\R^4),
\end{equation}
where $\operatorname{div}_S X$ is the trace of $DX$ restricted to the
plane $S$.  In general $\delta V$ is a distribution.  If it is locally
bounded with respect to the supremum norm on test fields, the
Riesz--Markov theorem, applied componentwise, represents it by an
$\R^4$-valued Radon measure, still denoted by $\delta V$; its scalar
total-variation measure is denoted by $\lVert\delta V\rVert$.  For such
a varifold, the generalized mean-curvature vector $H_V^{\R^4}$ is minus the
$\R^4$-valued Radon--Nikodym derivative of the part of $\delta V$ that
is absolutely continuous with respect to $\lVert V\rVert$.  In the
case considered below, the singular part vanishes and therefore
\begin{equation}
 \label{eq:generalized-mean-curvature-definition}
 \delta V(X)
 =-\int_{\R^4}\langle H_V^{\R^4},X\rangle\dd\lVert V\rVert.
\end{equation}
For any varifold $\Var$ we use the standard abbreviation
$\spt\Var:=\spt\lVert\Var\rVert$.
We regard $\lambda_t$ as a Radon measure on $\Sph^3$ in the transport
calculation and, via the inclusion $\Sph^3\hookrightarrow\R^4$, as the
weight measure of $\Var_t$ in the compactness argument.
This definition keeps the natural multiplicity when $F_t$ is not
injective.  For a countably two-rectifiable set $M$ and a multiplicity
$\theta$, let $v(M,\theta)$ denote the rectifiable varifold carried by
$M$ with multiplicity $\theta$; let $\mathcal H^2$ denote
two-dimensional Hausdorff measure and $\mathcal H^2\llcorner M$ its
restriction to $M$.  If $F_t$ is an embedding, the area formula gives
\begin{equation}
 \label{eq:embedded-varifold}
 \Var_t=v(F_t(\Sigma),1),
 \qquad
 \lambda_t=\mathcal H^2\llcorner F_t(\Sigma).
\end{equation}

For $\varphi\in C^1(\Sph^3)$, put
\[
 \lVert\varphi\rVert_{C^1(\Sph^3)}
 :=\lVert\varphi\rVert_{L^\infty(\Sph^3)}
   +\lVert\nabla^{\Sph^3}\varphi\rVert_{L^\infty(\Sph^3)}.
\]
For positive Radon measures $\alpha$ and $\beta$ on $\Sph^3$, we use
the bounded-Lipschitz distance
\begin{equation}
 \label{eq:bl-distance}
 d_{\BL}(\alpha,\beta)
 :=\sup\left\{
 \left|\int\varphi\dd\alpha-\int\varphi\dd\beta\right|:
 \varphi\in C^1(\Sph^3),\
 \lVert\varphi\rVert_{C^1(\Sph^3)}\leq1
 \right\}.
\end{equation}
For nonempty compact sets $K,L\subset\R^4$, the Hausdorff distance
used below is the Euclidean Hausdorff distance
\begin{equation}
 \label{eq:hausdorff-distance}
 d_{\mathcal H}(K,L)
 :=\max\left\{
 \sup_{x\in K}\inf_{y\in L}|x-y|,
 \sup_{y\in L}\inf_{x\in K}|x-y|
 \right\}.
\end{equation}
Finally, $\rightharpoonup$ denotes weak convergence of Radon measures
or weak varifold convergence, according to context; the latter means
convergence against every test function in $C_c(G_2(\R^4))$.
We abbreviate $L^2(\Sigma,\dd\mu_t)$ to $L^2(\dd\mu_t)$ whenever the
domain is clear.

\section{Intrinsic and image-measure estimates}

\begin{lemma}[Control of the geometric velocity]
\label{lem:velocity}
Assume \eqref{eq:un}.  Then, for every $0<t<T$,
\begin{equation}
 \label{eq:velocity}
 \lVert\partial_tF_t\rVert_{L^2(\Sigma,\dd\mu_t)}^2
 \leq c_0^{-2}
 \bigl(-W'(t)\bigr).
\end{equation}
\end{lemma}

\begin{proof}
By \eqref{eq:miwf},
\[
 \begin{aligned}
 \lVert\partial_tF_t\rVert_{L^2(\dd\mu_t)}^2
 &=\int_\Sigma a_t^{-4}\lvert G_t\rvert^2\dd\mu_t\\
 &=\int_\Sigma a_t^{-2}
       \bigl(a_t^{-2}\lvert G_t\rvert^2\bigr)\dd\mu_t.
 \end{aligned}
\]
Since $a_t\geq c_0$, one has $a_t^{-2}\leq c_0^{-2}$, and the
conclusion follows from \eqref{eq:dissipation}.
\end{proof}

\begin{lemma}[Evolution of the area and its push-forward]
\label{lem:transport}
Let $X_t:=\partial_tF_t$.  Then the induced area measure on the fixed
domain satisfies
\begin{equation}
 \label{eq:domain-evolution}
 \partial_t\dd\mu_t
 =-\langle H_t,X_t\rangle\dd\mu_t
 =a_t^{-2}\langle H_t,G_t\rangle\dd\mu_t.
\end{equation}
Moreover, for every $\varphi\in C^1(\Sph^3)$,
\begin{equation}
 \label{eq:pushforward-evolution}
 \begin{aligned}
 \frac{\dd}{\dd t}\int_{\Sph^3}\varphi\dd\lambda_t
 &=\int_\Sigma
 \left[
  \big\langle\nabla^{\Sph^3}\varphi(F_t),X_t\big\rangle
  -\varphi(F_t)\langle H_t,X_t\rangle
 \right]\dd\mu_t\\
 &=\int_\Sigma a_t^{-2}
 \left[
  -\big\langle\nabla^{\Sph^3}\varphi(F_t),G_t\big\rangle
  +\varphi(F_t)\langle H_t,G_t\rangle
 \right]\dd\mu_t.
 \end{aligned}
\end{equation}
\end{lemma}

\begin{proof}
The MIWF velocity is normal to the immersed surface inside $\Sph^3$.
The first variation of area for a closed immersed surface in $\Sph^3$
therefore gives the first identity in \eqref{eq:domain-evolution}; the
second follows from $X_t=-a_t^{-2}G_t$.  By the definition of the
push-forward measure,
\[
 \int_{\Sph^3}\varphi\dd\lambda_t
 =\int_\Sigma\varphi(F_t)\dd\mu_t.
\]
Differentiating this identity, applying the chain rule to
$\varphi(F_t)$, and using \eqref{eq:domain-evolution} gives the first
line of \eqref{eq:pushforward-evolution}.  Substituting the MIWF
equation gives the second line.  Thus both the motion of the immersion
and the variation of its induced area measure are accounted for.
\end{proof}

We regard $\mu_t$ as the area measure on the fixed domain, so that
\[
 \mu_t(B)=\int_B\dd\mu_t,
 \qquad B\subset\Sigma\ \text{Borel}.
\]
For a signed Radon measure $\sigma$ on $\Sigma$, we use the convention
\[
 \lVert\sigma\rVert_{\TV}
 :=\lvert\sigma\rvert(\Sigma)
 =\sup_{\substack{\psi\in C^0(\Sigma)\\
                   \lVert\psi\rVert_{L^\infty}\leq1}}
   \left|\int_\Sigma\psi\dd\sigma\right|.
\]

The area-evolution and energy-dissipation argument underlying the following estimate also appears in the proof of Lemma D.6 of \cite{DallAcquaMuellerSchaetzleSpener} for the classical Willmore flow. 
Here we adapt this argument to the M\"obius-invariant Willmore flow under condition \eqref{eq:un} and estimate the induced measures in total variation.
\begin{proposition}[Total-variation estimate on the fixed domain]
\label{prop:fixed-tv}
Assume \eqref{eq:un}, with the constant $c_0>0$ fixed there.  For every
$0<s<t<T$,
\begin{equation}
 \label{eq:fixed-tv}
 \lVert\mu_t-\mu_s\rVert_{\TV}
 \leq
 \frac{2\sqrt{W_0}}{c_0}
 \sqrt{(t-s)\bigl(W(s)-W(t)\bigr)}.
\end{equation}
In particular, for every Borel set $B\subset\Sigma$,
\begin{equation}
 \label{eq:borel-tv}
 \lvert\mu_t(B)-\mu_s(B)\rvert
 \leq
 \frac{2\sqrt{W_0}}{c_0}
 \sqrt{(t-s)\bigl(W(s)-W(t)\bigr)}.
\end{equation}
If $W_0\leq8\pi$, the constant $2\sqrt{W_0}/c_0$ in
\eqref{eq:fixed-tv}--\eqref{eq:borel-tv} may be replaced by
$\sqrt{32\pi}/c_0$.
\end{proposition}

\begin{proof}
Fix a smooth positive reference measure $\dd\bar\mu$ on $\Sigma$ and
write $\dd\mu_r=\rho_r\dd\bar\mu$.  From
\eqref{eq:domain-evolution} and the
fundamental theorem of calculus,
\[
 \begin{aligned}
 \lVert\mu_t-\mu_s\rVert_{\TV}
 &=\int_\Sigma\lvert\rho_t-\rho_s\rvert\dd\bar\mu\\
 &\leq\int_s^t\int_\Sigma
 a_r^{-2}\lvert H_r\rvert\lvert G_r\rvert
 \dd\mu_r\dd r.
 \end{aligned}
\]
For each $r$, Cauchy--Schwarz, \eqref{eq:un}, and the identities
\[
 \int_\Sigma\lvert H_r\rvert^2\dd\mu_r
 =4\bigl(W(r)-\Area(F_r)\bigr)\leq4W(r),
 \qquad
 \int_\Sigma a_r^{-2}\lvert G_r\rvert^2\dd\mu_r=-W'(r),
\]
give
\[
 \begin{aligned}
 \int_\Sigma a_r^{-2}\lvert H_r\rvert\lvert G_r\rvert
 \dd\mu_r
 &\leq
 \left(\int_\Sigma a_r^{-2}\lvert H_r\rvert^2\dd\mu_r\right)^{1/2}
 \left(\int_\Sigma a_r^{-2}\lvert G_r\rvert^2\dd\mu_r\right)^{1/2}\\
 &\leq
 \frac{2\sqrt{W(r)}}{c_0}
 \sqrt{-W'(r)}\\
 &\leq
 \frac{2\sqrt{W_0}}{c_0}\sqrt{-W'(r)}.
 \end{aligned}
\]
A second application of Cauchy--Schwarz, now in time, yields
\[
 \begin{aligned}
 \lVert\mu_t-\mu_s\rVert_{\TV}
 &\leq\frac{2\sqrt{W_0}}{c_0}
 \int_s^t\sqrt{-W'(r)}\dd r\\
 &\leq\frac{2\sqrt{W_0}}{c_0}
 \sqrt{t-s}
 \sqrt{\int_s^t-W'(r)\dd r},
 \end{aligned}
\]
which is \eqref{eq:fixed-tv}.  The estimate for a Borel set follows from
$\lvert(\mu_t-\mu_s)(B)\rvert\leq
\lVert\mu_t-\mu_s\rVert_{\TV}$.
\end{proof}

The preceding estimate controls measures on the fixed domain.  To
control their push-forwards one must also estimate the motion of the
maps $F_t$.

\begin{proposition}[Bounded-Lipschitz estimate for the image measures]
\label{prop:image-bl}
Under \eqref{eq:un}, for every $0<s<t<T$ and every
$\varphi\in C^1(\Sph^3)$,
\begin{equation}
 \label{eq:image-bl}
 \left|\int_{\Sph^3}\varphi\dd\lambda_t
       -\int_{\Sph^3}\varphi\dd\lambda_s\right|
 \leq
 \frac{2\sqrt{W_0}}{c_0}\,
 \lVert\varphi\rVert_{C^1(\Sph^3)}
 \sqrt{(t-s)\bigl(W(s)-W(t)\bigr)},
\end{equation}
where the $C^1(\Sph^3)$-norm is the one fixed above
\eqref{eq:bl-distance}.
\end{proposition}

\begin{proof}
Lemma~\ref{lem:transport} and Cauchy--Schwarz give
\[
 \begin{aligned}
 \left|\frac{\dd}{\dd t}
       \int_{\Sph^3}\varphi\dd\lambda_t\right|
 &\leq
 \Bigl(
   \lVert\nabla^{\Sph^3}\varphi\rVert_{L^\infty}
    \Area(F_t)^{1/2}
   +\lVert\varphi\rVert_{L^\infty}
    \lVert H_t\rVert_{L^2(\dd\mu_t)}
 \Bigr)
 \lVert X_t\rVert_{L^2(\dd\mu_t)}.
 \end{aligned}
\]
Since
\[
 \Area(F_t)\leq W(t)\leq W_0,
 \qquad
 \lVert H_t\rVert_{L^2(\dd\mu_t)}
 =2\sqrt{W(t)-\Area(F_t)}\leq2\sqrt{W_0},
\]
Lemma~\ref{lem:velocity} gives
\[
 \left|\frac{\dd}{\dd t}
       \int_{\Sph^3}\varphi\dd\lambda_t\right|
 \leq
 \frac{2\sqrt{W_0}}{c_0}
 \lVert\varphi\rVert_{C^1(\Sph^3)}\sqrt{-W'(t)}.
\]
Integration over $[s,t]$ and Cauchy--Schwarz in time prove
\eqref{eq:image-bl}.
\end{proof}

\begin{remark}[Role of the two estimates and of the ambient sphere]
The area-evolution argument leading to \eqref{eq:fixed-tv} and the
push-forward calculation \eqref{eq:pushforward-evolution} play
different roles.  The latter is the decisive estimate for the
geometric images: convergence of the intrinsic measures $\mu_t$
on $\Sigma$ does not by itself imply convergence of
$(F_t)_\#\mu_t$, because the maps $F_t$ vary with time.  The transport
term involving $\nabla^{\Sph^3}\varphi(F_t)$ is therefore essential.

The entire calculation is intrinsic to $\Sph^3$.  In particular,
\[
 \Area(F_t)\leq W(t)\leq W_0,
\]
and, when $W_0\leq8\pi$, the area is uniformly bounded by $8\pi$.
No stereographic projection or choice of a projection point is used.
The inclusion $\Sph^3\subset\R^4$ enters only later, when the image
measures are regarded as weights of varifolds and Allard's theorem and
Theorem~1.1 of \cite{Jakob} are applied.
\end{remark}

\begin{remark}[The role of finite time]
Estimate \eqref{eq:fixed-tv} is valid for arbitrary pairs of times.
When $T<\infty$, it implies that $(\mu_t)_{t<T}$ is Cauchy in total
variation as $t\nearrow T$.  The proof of Theorem~\ref{thm:main} uses
the image-measure estimate \eqref{eq:image-bl} together with this
finiteness of $T$; the fixed-domain estimate is not a substitute for
the transport estimate.  When $T=\infty$, the factor $t-s$ in these
estimates is unbounded and the estimates alone do not imply a Cauchy
property.  Proposition~\ref{prop:infinite-finite-length} below records
the additional finite-dissipation-length condition under which the
argument does extend to infinite time.
\end{remark}

\section{Proof of the unique-limit theorem}

\begin{lemma}[A rectifiable varifold is determined by its weight]
\label{lem:weight-determines-varifold}
Let $V_1$ and $V_2$ be rectifiable two-varifolds in $\R^4$.  If
$\lVert V_1\rVert=\lVert V_2\rVert$, then $V_1=V_2$.
\end{lemma}

\begin{proof}
Write the common weight as
\[
 \lambda=\theta\,\mathcal H^2\llcorner M
\]
with $M$ countably two-rectifiable and $\theta>0$ locally
$\mathcal H^2$-integrable.  The approximate tangent plane
$T_x\lambda=T_xM$ is uniquely determined for $\lambda$-almost every
$x$.  By rectifiability, the plane component of each $V_i$ is the
Dirac mass at $T_x\lambda$ for $\lambda$-almost every $x$.  Hence
\[
 V_i(\Phi)=\int_M\Phi(x,T_xM)\,\theta(x)\dd\mathcal H^2(x)
 \qquad(i=1,2)
\]
for every $\Phi\in C_c(G_2(\R^4))$.  Thus $V_1=V_2$; compare the
rectifiable-varifold representation in \cite{Allard}.
\end{proof}

\begin{proof}[Proof of Theorem~\ref{thm:main}]
Proposition~\ref{prop:fixed-tv} and the finiteness of $T$ show that
$(\mu_t)_{t<T}$ is Cauchy in total variation.  The space of signed
Radon measures on $\Sigma$ is complete in the total-variation norm,
and its nonnegative cone is closed.  Hence there is a Radon
measure $\mu_T$ such that $\mu_t\to\mu_T$ in total variation.
Letting $t\nearrow T$ in \eqref{eq:fixed-tv} gives
\eqref{eq:main-endpoint-tv}.

Proposition~\ref{prop:image-bl} implies
\begin{equation}
 \label{eq:bl-cauchy}
 d_{\BL}(\lambda_s,\lambda_t)
 \leq\frac{2\sqrt{W_0}}{c_0}
 \sqrt{(t-s)\bigl(W(s)-W(t)\bigr)}.
\end{equation}
Because $T<\infty$, the right-hand side tends to zero as
$s,t\nearrow T$.  Moreover,
\[
 \lambda_t(\Sph^3)=\Area(F_t)\leq W_0.
\]
Since $\Sph^3$ is compact, the Radon measures $\lambda$ with
$\lambda(\Sph^3)\leq W_0$ form a sequentially compact set for weak-*
convergence in $C(\Sph^3)^*$.  Hence every sequence $t_j\nearrow T$
has a subsequence along which $\lambda_{t_j}$ converges weak-* in
$C(\Sph^3)^*$.  On this uniformly bounded family the bounded-Lipschitz
distance metrizes this weak-* convergence, and estimate
\eqref{eq:bl-cauchy} forces
any two such cluster measures to coincide.  It follows that there is a
unique Radon measure $\lambda_T$ and that the full family
satisfies
\begin{equation}
 \label{eq:weight-limit}
 \lambda_t\rightharpoonup\lambda_T
 \qquad\text{as }t\nearrow T.
\end{equation}
In fact the convergence holds in $d_{\BL}$, and letting
$t\nearrow T$ in \eqref{eq:bl-cauchy} proves
\eqref{eq:main-endpoint-bl}.

We next use varifold compactness.  For every $t$, $\Var_t$ is an
integral varifold supported in $\Sph^3$ and
\[
 \lVert\Var_t\rVert(\R^4)=\Area(F_t)\leq W_0.
\]
Since $\Var_t$ is induced by a smooth closed immersion, its first
variation is represented by the Euclidean mean-curvature vector and
satisfies
\begin{equation}
 \label{eq:first-variation}
 \begin{aligned}
 \lVert\delta\Var_t\rVert(\R^4)
 &\leq\int_\Sigma\lvert H_t^{\R^4}\rvert\dd\mu_t\\
 &\leq
 \lVert H_t^{\R^4}\rVert_{L^2(\dd\mu_t)}
 \Area(F_t)^{1/2}
 \leq2W_0.
 \end{aligned}
\end{equation}
Allard's integral compactness theorem \cite{Allard} therefore shows
that every sequence $t_j\nearrow T$ admits a subsequence $t_{j_k}$ and
an integral two-varifold $\Var$ such that
\begin{equation}
 \label{eq:cluster-varifold}
 \Var_{t_{j_k}}\rightharpoonup\Var.
\end{equation}
Taking weight measures in \eqref{eq:cluster-varifold} and using
\eqref{eq:weight-limit} gives
\begin{equation}
 \label{eq:cluster-weight}
 \lVert\Var\rVert=\lambda_T.
\end{equation}

If $\lambda_T=0$, set $\Var_T:=0$; then
\eqref{eq:cluster-weight} makes every cluster varifold the zero
varifold.  If $\lambda_T\neq0$, all cluster varifolds are integral,
and hence rectifiable, by Allard compactness and have the same weight
$\lambda_T$ by \eqref{eq:cluster-weight}.  Lemma
\ref{lem:weight-determines-varifold} therefore shows that they
coincide; denote their common value by $\Var_T$.

If the full family failed to converge to $\Var_T$, one could choose a
certain sequence $t_j\nearrow T$ such that the varifolds $\Var_{t_{j}}$
would stay outside a fixed weak-varifold neighborhood of $\Var_T$.  
Varifold compactness would then produce a subsequence of $\{\Var_{t_{j}}\}$
converging to a different cluster varifold, thus a contradiction.  This
proves \eqref{eq:main-convergence}.

It remains to prove the endpoint curvature assertion.  For every
$X\in C_c^1(\R^4;\R^4)$, the definition
\eqref{eq:first-variation-definition} and varifold convergence give
\[
 \delta\Var_t(X)\longrightarrow\delta\Var_T(X),
\]
because $(x,S)\mapsto\operatorname{div}_S X(x)$ is continuous and
compactly supported on $G_2(\R^4)$.  Since
\[
 \delta\Var_t(X)
 =-\int_\Sigma
 \langle H_t^{\R^4},X(F_t)\rangle\dd\mu_t,
\]
Cauchy--Schwarz, \eqref{eq:euclidean-willmore},
\eqref{eq:weight-limit}, and $W(t)\to W_T$ imply
\begin{align*}
 |\delta\Var_T(X)|
 &\leq
 \limsup_{t\nearrow T}
 \left(\int_\Sigma|H_t^{\R^4}|^2\dd\mu_t\right)^{1/2}
 \left(\int_{\Sph^3}|X|^2\dd\lambda_t\right)^{1/2}\\
 &=2\sqrt{W_T}
 \left(\int_{\Sph^3}|X|^2\dd\lambda_T\right)^{1/2}.
\end{align*}
Since $C_c^1(\R^4;\R^4)$ is dense in
$L^2(\lambda_T;\R^4)$, the preceding estimate shows that
$\delta\Var_T$ extends uniquely to a bounded linear functional on this
space.  The Riesz representation theorem for this Hilbert space
supplies a vector field
$H_{\Var_T}^{\R^4}\in L^2(\lambda_T;\R^4)$ satisfying
\[
 \delta\Var_T(X)
 =-\int_{\R^4}\langle H_{\Var_T}^{\R^4},X\rangle\dd\lambda_T,
 \qquad
 \lVert H_{\Var_T}^{\R^4}\rVert_{L^2(\lambda_T)}
 \leq2\sqrt{W_T}.
\]
By \eqref{eq:generalized-mean-curvature-definition}, this vector field
is the generalized Euclidean mean curvature of $\Var_T$, the singular
part of $\delta\Var_T$ vanishes, and the norm bound is precisely
\eqref{eq:main-curvature-lsc}.  This is the standard
lower-semicontinuity mechanism for weak mean curvature under varifold
convergence; compare \cite[(1.1) and the introduction]
{SchaetzleLowerSemicontinuity}.
\end{proof}

\subsection{A finite-dissipation-length criterion at infinite time}

The finite-time hypothesis in Theorem~\ref{thm:main} can be replaced
by an explicit integrability condition on the dissipated energy.

\begin{proposition}[Conditional unique limit at infinite time]
\label{prop:infinite-finite-length}
Let $F\colon[0,\infty)\times\Sigma\to\Sph^3$ be a smooth MIWF of a
compact torus.  Assume that there is one constant $c_0>0$ such that
\[
 \inf_{t>0}\min_{x\in\Sigma}|A_{F_t}^0(x)|^2\geq c_0,
\]
and, in addition, that
\begin{equation}
 \label{eq:finite-dissipation-length}
 \int_0^\infty\sqrt{-W'(r)}\dd r<\infty.
\end{equation}
Then there are positive Radon measures $\mu_\infty$ on $\Sigma$ and
$\lambda_\infty$ on $\Sph^3$ and a unique integral two-varifold
$\Var_\infty$ in $\R^4$, possibly the zero varifold, such that
\[
 \mu_t\longrightarrow\mu_\infty\quad\text{in total variation},
 \qquad
 \lambda_t\longrightarrow\lambda_\infty
 \quad\text{in }d_{\BL},
 \qquad
 \Var_t\rightharpoonup\Var_\infty
\]
as $t\to\infty$, and $\lambda_\infty=\lVert\Var_\infty\rVert$.
Writing $W_\infty:=\lim_{t\to\infty}W(t)$, one also has
\[
 H_{\Var_\infty}^{\R^4}
 \in L^2(\lVert\Var_\infty\rVert;\R^4),
 \qquad
 \frac14\int_{\R^4}|H_{\Var_\infty}^{\R^4}|^2
 \dd\lVert\Var_\infty\rVert\leq W_\infty.
\]
\end{proposition}

\begin{proof}
Before the final Cauchy--Schwarz inequality in time, the proofs of
Propositions~\ref{prop:fixed-tv} and~\ref{prop:image-bl} give, for
$0<s<t$,
\begin{align*}
 \lVert\mu_t-\mu_s\rVert_{\TV}
 &\leq\frac{2\sqrt{W_0}}{c_0}
 \int_s^t\sqrt{-W'(r)}\dd r,\\
 d_{\BL}(\lambda_t,\lambda_s)
 &\leq\frac{2\sqrt{W_0}}{c_0}
 \int_s^t\sqrt{-W'(r)}\dd r.
\end{align*}
Assumption \eqref{eq:finite-dissipation-length} makes both right-hand
sides tend to zero as $s,t\to\infty$.  Completeness of the signed
Radon measures on $\Sigma$ in total variation gives $\mu_\infty$.
Compactness of $\Sph^3$, the uniform mass bound
$\lambda_t(\Sph^3)\leq W_0$, and the second Cauchy estimate give the
unique limit $\lambda_\infty$ in $d_{\BL}$.

The mass bound $\lVert\Var_t\rVert(\R^4)\leq W_0$ and the
first-variation bound \eqref{eq:first-variation} are uniform for all
$t>0$.  Allard
compactness therefore makes every sequence $t_j\to\infty$ admit an
integral-varifold cluster limit.  Each such limit has weight
$\lambda_\infty$, so Lemma~\ref{lem:weight-determines-varifold} shows
that all cluster varifolds coincide.  The usual contradiction argument
then gives convergence of the full family to a unique integral
varifold $\Var_\infty$ with weight $\lambda_\infty$.

Finally, $W(t)$ decreases to $W_\infty$.  Applying
\eqref{eq:first-variation-definition} along an arbitrary sequence
$t_j\to\infty$ and repeating the $L^2(\lambda_\infty)$ estimate in the
last part of the proof of Theorem~\ref{thm:main} gives
\[
 |\delta\Var_\infty(X)|
 \leq2\sqrt{W_\infty}
 \lVert X\rVert_{L^2(\lambda_\infty)}.
\]
The Hilbert-space Riesz theorem and
\eqref{eq:generalized-mean-curvature-definition} prove the asserted
mean-curvature statement and its bound.
\end{proof}

\begin{remark}[Relation to a \L ojasiewicz--Simon argument]
Condition \eqref{eq:finite-dissipation-length} is not a consequence of
the hypotheses used elsewhere in this paper.  A
\L ojasiewicz--Simon gradient inequality could imply such a tail
condition only after one proves that, in a fixed gauge, the trajectory
eventually enters and remains in a neighborhood of one smooth critical
immersion.  Neither this trapping statement nor the required analytic
gauge construction follows from the subsequential compactness theorem
of \cite{Jakob}.  Accordingly, Proposition
\ref{prop:infinite-finite-length} is conditional; an unconditional
infinite-time uniqueness theorem is not claimed here.
\end{remark}

\section{Unique endpoint alternatives at finite maximal time}

The results in this section concern the maximal time of existence,
because this is the endpoint to which the compactness theorems of
\cite{Jakob} apply.  We combine Theorem~\ref{thm:main} directly with those results
in $\Sph^3$; no stereographic projection is introduced.

\begin{lemma}[Umbilic event versus uniform nonumbilicity]
\label{lem:umbilic-alternative}
Let $F\colon[0,T)\times\Sigma\to\Sph^3$ be a smooth MIWF with
$T<\infty$.  Then exactly one of the following alternatives holds:
\begin{enumerate}[label=\textup{(\roman*)}]
 \item there are $t_j\nearrow T$ and $x_j\in\Sigma$ such that
 \begin{equation}
  \label{eq:umbilic-event}
  |A^0_{F_{t_j}}(x_j)|^2\longrightarrow0;
 \end{equation}
 \item there is $c_0>0$ such that
 \begin{equation}
  \label{eq:uniform-nonumbilic-alternative}
  \inf_{0<t<T}\min_{x\in\Sigma}|A^0_{F_t}(x)|^2\geq c_0.
 \end{equation}
\end{enumerate}
\end{lemma}

\begin{proof}
Set
\[
 m(t):=\min_{x\in\Sigma}|A^0_{F_t}(x)|^2.
\]
The function $m$ is continuous and positive on $[0,T)$, since the
flow is smooth and each $F_t$ is umbilic-free.  Consequently, for every
$\varepsilon>0$ it has a positive minimum on the compact interval
$[0,T-\varepsilon]$.  If
\eqref{eq:uniform-nonumbilic-alternative} fails, choose
$\tau_j\in(0,T)$ with $m(\tau_j)<1/j$.  The preceding compact-interval
bound forces $\tau_j\to T$.  After passing to an increasing subsequence
and choosing $x_j$ at which $m(\tau_j)$ is attained, one obtains
\eqref{eq:umbilic-event}.  Conversely,
\eqref{eq:uniform-nonumbilic-alternative} plainly excludes
\eqref{eq:umbilic-event}.
\end{proof}

\begin{corollary}[Sequence-independent endpoint topology]
\label{cor:classification}
Assume the hypotheses of Theorem~\ref{thm:main} with
\begin{equation}
 \label{eq:maximal-time}
 T=T_{\max}(F_0)<\infty
 \qquad\text{and}\qquad
 W(0)\leq8\pi.
\end{equation}
Then exactly one of the following mutually exclusive alternatives
holds:
\begin{enumerate}[label=\textup{(\alph*)}]
 \item $\Var_T=0$;
 \item $\Var_T\neq0$ and $\spt\Var_T$ is a closed embedded orientable
 Lipschitz surface of genus zero;
 \item $\Var_T\neq0$ and $\spt\Var_T$ is a closed embedded orientable
 Lipschitz surface of genus one.
\end{enumerate}
In the nonzero alternatives, $\Var_T$ has unit density.
The alternative and the limiting varifold are independent of the
sequence of times tending to $T$.
\end{corollary}

\begin{proof}
If $W(t_0)=8\pi$ for some $t_0>0$, monotonicity and $W(0)\leq8\pi$
give $W(t)=8\pi$ on $[0,t_0]$.  The dissipation identity then implies
$G_t=0$ for every $t\in(0,t_0)$, and hence $F_t$ is stationary on that
interval.  By continuity it is stationary up to $t=t_0$.  The stationary
solution is a smooth global solution, and uniqueness in the short-time
theory for the MIWF \cite{JakobShortTime} therefore makes it the
maximal continuation of $F_{t_0}$.  This
contradicts $T_{\max}(F_0)<\infty$.
Consequently,
\[
 W(t)<8\pi\qquad\text{for every }0<t<T,
\]
and the Li--Yau inequality \cite{LiYau}, as used in
\cite[proof of Theorem~1.1]{Jakob}, shows that every $F_t$, $t>0$, is
an embedding.
Thus \eqref{eq:embedded-varifold} holds.

Given any sequence $t_j\nearrow T$, Theorem~1.1(1) of \cite{Jakob}
provides a subsequence and an integral two-varifold
$\widetilde{\Var}$ such that the multiplicity-one image measures
converge weakly to $\lVert\widetilde{\Var}\rVert$.  If this limit is
nonzero, the same theorem gives unit density and says that its support is
a closed embedded orientable Lipschitz surface of genus zero or one.
Since these image measures are precisely $\lambda_t$, one has
$\lVert\widetilde{\Var}\rVert=\lambda_T
=\lVert\Var_T\rVert$.  Both varifolds are rectifiable, so
Lemma~\ref{lem:weight-determines-varifold} gives
$\widetilde{\Var}=\Var_T$.
This proves the endpoint trichotomy and its sequence independence.
\end{proof}

\begin{corollary}[Exhaustive finite-time breakdown alternative]
\label{cor:four-way}
Let $F\colon[0,T_{\max})\times\Sigma\to\Sph^3$ be a maximal smooth
MIWF with
\[
 T_{\max}=T_{\max}(F_0)<\infty,
 \qquad W(0)\leq8\pi.
\]
Then exactly one of the following mutually exclusive alternatives
holds:
\begin{enumerate}[label=\textup{(\Alph*)}]
 \item there are $t_j\nearrow T_{\max}$ and $x_j\in\Sigma$ such that
 \[
  |A^0_{F_{t_j}}(x_j)|^2\longrightarrow0;
 \]
 \item uniform nonumbilicity holds and the unique limit is zero:
 $\Var_{T_{\max}}=0$;
 \item uniform nonumbilicity holds, $\Var_{T_{\max}}\neq0$, and
 $\spt\Var_{T_{\max}}$ is a closed embedded orientable Lipschitz
 surface of genus zero;
 \item uniform nonumbilicity holds, $\Var_{T_{\max}}\neq0$, and
 $\spt\Var_{T_{\max}}$ is a closed embedded orientable Lipschitz
 surface of genus one.
\end{enumerate}
In alternatives \textup{(C)} and \textup{(D)}, the limit has unit
density.  No assertion made here excludes alternative \textup{(D)}.
\end{corollary}

\begin{proof}
Lemma~\ref{lem:umbilic-alternative} gives the mutually exclusive
alternative between \textup{(A)} and uniform nonumbilicity.  In the
latter case Theorem~\ref{thm:main} and
Corollary~\ref{cor:classification} give exactly \textup{(B)},
\textup{(C)}, or \textup{(D)}, including the unit-density statement.
\end{proof}

\begin{corollary}[Full Hausdorff convergence in the nonzero case]
\label{cor:hausdorff}
Under the hypotheses of Corollary~\ref{cor:classification}, if
$\Var_T\neq0$, then
\begin{equation}
 \label{eq:hausdorff}
 F_t(\Sigma)\longrightarrow\spt\Var_T
 \qquad\text{in Hausdorff distance as }t\nearrow T.
\end{equation}
\end{corollary}

\begin{proof}
If \eqref{eq:hausdorff} failed, there would exist $\varepsilon>0$ and a
sequence $t_j\nearrow T$ such that
\[
 d_{\mathcal H}\bigl(F_{t_j}(\Sigma),\spt\Var_T\bigr)
 \geq\varepsilon.
\]
After passing to a subsequence, Theorem~1.1(1) of \cite{Jakob} gives
weak convergence of the image measures and, because the limit is
nonzero, Hausdorff convergence to the support of the limiting
varifold.  The latter is $\Var_T$ by Theorem~\ref{thm:main}, a contradiction.
\end{proof}

\begin{corollary}[Parametrization of the unique toroidal limit]
\label{cor:parametrization}
Assume the hypotheses of Corollary~\ref{cor:classification} and suppose that
$\Var_T\neq0$ and $\spt\Var_T$ has genus one.  Then there exists a
uniformly conformal bi-Lipschitz homeomorphism
\[
 f\in W^{2,2}(\Sigma,\R^4)\cap W^{1,\infty}(\Sigma,\R^4),
 \qquad
 f\colon\Sigma\longrightarrow M:=\spt\Var_T,
\]
and there are a smooth zero-scalar-curvature, unit-volume metric
$g_{\mathrm{poin}}$ on $\Sigma$ and a function $u\in L^\infty(\Sigma)$
such that
\begin{equation}
 \label{eq:conformal-limit}
 f^*g_{\mathrm{euc}}=e^{2u}g_{\mathrm{poin}}.
\end{equation}
This identity is understood almost everywhere on \(\Sigma\).
Moreover,
\begin{equation}
 \label{eq:unit-density-limit}
 \Var_T=v(M,1),
\end{equation}
and
\begin{equation}
 \label{eq:limit-measure}
 \lVert\Var_T\rVert
 =f_\#\mu_{f^*g_{\mathrm{euc}}}
 =\mathcal H^2\llcorner M.
\end{equation}
In particular, formula \eqref{eq:limit-measure} is independent of the sequence
used to construct the parametrization $f$.
\end{corollary}

\begin{proof}
Apply Theorem~1.1(2) and formula (9) of \cite{Jakob} to any sequence
realizing the nonzero genus-one limit.  They provide $f$, the conformal
identity \eqref{eq:conformal-limit}, and the measure identity
\eqref{eq:limit-measure}.
Since the limit has unit density, the corresponding full rectifiable
varifold is $v(M,1)$.  Equivalently, it is obtained by integrating over
$\Sigma$ with the area measure of $f^*g_{\mathrm{euc}}$ and with the
approximate tangent plane carried by $f$.  Theorem~\ref{thm:main} makes this
geometric varifold independent of the sequence.
\end{proof}

\begin{remark}[What is and is not unique]
Theorem~\ref{thm:main} and Corollary~\ref{cor:parametrization} prove uniqueness of the unparametrized
geometric limit.  They do not produce a unique limiting map
$F_T\colon\Sigma\to\Sph^3$.  The parametrizations in
\cite[Theorem~1.1(2)]{Jakob} are obtained after sequence-dependent
diffeomorphisms of the domain.  Different choices can therefore give
different parametrizations of the same uniquely determined varifold.
\end{remark}

\begin{remark}[Nonumbilicity is not stable under the available convergence]
\label{rem:nonumbilic-transfer}
The uniform lower bound along the smooth trajectory does not, by
itself, imply
\[
 \operatorname*{ess\,inf}_{x\in\Sigma}|A^0_f(x)|^2>0
\]
for the parametrization supplied by
Corollary~\ref{cor:parametrization}, where $A_f^0$ denotes the weak
trace-free second fundamental form of $f$.  The convergence in
\cite[Theorem~1.1(2)]{Jakob} is only weak in $W^{2,2}$ and weak-* in
$W^{1,\infty}$ after sequence-dependent reparametrizations.  Since
$A^0$ depends on second derivatives and on the induced metric, neither
this weak convergence nor the accompanying $C^0$ convergence transmits
a pointwise lower bound.  Such a conclusion would require an
additional no-defect argument, for example a strong $W^{2,2}$
compactness statement with sufficient control of the induced metrics.
The stronger convergence in \cite[Theorem~1.1(3)]{Jakob} requires an
energy identity which is not among the hypotheses of
Corollary~\ref{cor:parametrization}.
\end{remark}

\subsection{The unresolved toroidal continuation problem}

Alternative \textup{(D)} in Corollary~\ref{cor:four-way} cannot be
removed by the results proved in this paper.  Turning the unique
toroidal endpoint into a continuation criterion would require all of
the following additional ingredients.
\begin{enumerate}[label=\textup{(C\arabic*)}]
 \item A stability theorem transferring quantitative nonumbilicity to
 the limiting toroidal parametrization, as explained in
 Remark~\ref{rem:nonumbilic-transfer}.
 \item A time-dependent family of domain diffeomorphisms placing the
 entire trajectory in a single gauge and producing a trace at
 $t=T_{\max}$ in $W^{2,2}(\Sigma,\R^4)$.  Subsequence-dependent spatial
 reparametrizations do not provide such a trace.
 \item Short-time existence and uniqueness for quantitatively
 umbilic-free initial data in $W^{2,2}$.  The maximal-regularity theory
 in \cite{JakobRegularity} works in
 \[
  W^{1,p}(0,T;L^p)\cap L^p(0,T;W^{4,p}),
  \qquad
  F_0\in W^{4-4/p,p},
 \]
 with $p>3$.  Reaching $p=2$, proving a compatible trace theorem, and
 gluing the restarted solution to the original trajectory require a
 new critical-regularity theory.
\end{enumerate}
Only after \textup{(C1)}--\textup{(C3)} have been established can a
toroidal endpoint be ruled out by contradiction with the maximality of
$T_{\max}$.  The remaining finite-time breakdown mechanisms would then
be precisely the umbilic event, the zero endpoint, and the genus-zero
endpoint.  This conditional conclusion is not asserted as a theorem
here; compare \cite[Remarks~7.1--7.3]{Jakob}.

\section{Finite-time Hopf profiles and their conformal moduli}

Let \(\pi\colon\Sph^3\to\Sph^2\) denote the standard Hopf fibration.

\begin{corollary}[Finite-time Hopf torus trajectories]
\label{cor:hopf}
Suppose that the MIWF starts from a smooth parametrization \(F_0\) of
a smooth Hopf torus in the sense of Pinkall \cite{Pinkall}, that
\(W(0)\leq8\pi\), and that
\(T=T_{\max}(F_0)<\infty\).  Then there exists a unique unrenormalized
limiting integral varifold \(\Var_T\).  Its support is an embedded
\(C^1\) Hopf torus, and the convergence is both weak varifold
convergence and Hausdorff convergence.  Moreover,
\eqref{eq:conformal-limit}--\eqref{eq:limit-measure} hold
for a uniformly conformal bi-Lipschitz parametrization of the limiting
Hopf torus.  In particular, among the alternatives of
Corollary~\ref{cor:four-way}, every such finite-time Hopf trajectory
lies in alternative \textup{(D)}.
\end{corollary}

\begin{proof}
The argument in \cite[proof of Theorem~1.2(1)]{Jakob} combines
Proposition~4.4 there, existence and uniqueness for the reduced
degenerate elastic flow, uniqueness of the MIWF, and the
Hopf-fibration reduction developed in \cite{JakobHopf}.  It shows that
the trajectory remains in the Hopf sector, up to smooth time-dependent
domain reparametrization.  Thus each \(F_t\) parametrizes a Hopf torus.
For a Hopf torus with profile curve \(\gamma_t\),
Proposition~4.3 of \cite{Jakob} gives
\begin{equation}
 \label{eq:hopf-nonumbilic}
 |A_t^0|^2=2\bigl(1+\kappa_{\gamma_t}^2\bigr)\geq2 .
\end{equation}
Here \(\kappa_{\gamma_t}\) is the signed geodesic curvature of
\(\gamma_t\subset\Sph^2\).  Thus \eqref{eq:un} holds with \(c_0=2\).
Theorem~\ref{thm:main} yields uniqueness of the varifold limit, while
Theorem~1.2(1) of \cite{Jakob} shows that every subsequential limit in
the Hopf sector is supported on an embedded \(C^1\) Hopf torus.
Corollaries~\ref{cor:hausdorff} and~\ref{cor:parametrization} yield
the remaining claims.
\end{proof}

We next fix the profile gauge.  Identify \(\Sph^1\) with
\(\R/\mathbb Z\).  Choose the smooth oriented normal-gauge profile
\(\gamma_t\colon\Sph^1\to\Sph^2\) supplied by the reduced flow in
\cite[Proposition~4.4]{Jakob}, with a fixed material point
\(0\in\Sph^1\), and set
\begin{equation}
 \label{eq:elastic-energy}
 E(\gamma):=\int_{\Sph^1}(1+\kappa_\gamma^2)\dd s_\gamma,
 \qquad L(\gamma):=\int_{\Sph^1}\dd s_\gamma .
\end{equation}
Thus \(\partial_t\gamma_t\) is normal to the curve in \(\Sph^2\).
For \(x\in[0,1]\), extended equivariantly to \(\R\), define
\begin{equation}
 \label{eq:anchored-gauge}
 \phi_t(x):=\frac1{L(\gamma_t)}
       \int_0^x|\partial_z\gamma_t(z)|\dd z,\qquad
 \psi_t:=\phi_t^{-1},\qquad
 \widetilde\gamma_t:=\gamma_t\circ\psi_t .
\end{equation}
Then \(\phi_t(0)=\psi_t(0)=0\) and
\begin{equation}
 \label{eq:constant-profile-speed}
 |\partial_x\widetilde\gamma_t|=L(\gamma_t)
 \quad\text{on }\Sph^1 .
\end{equation}
Since \(\widetilde\gamma_t\) is an orientation-preserving
reparametrization of \(\gamma_t\), one also has
\[
 L(\widetilde\gamma_t)=L(\gamma_t),
 \qquad E(\widetilde\gamma_t)=E(\gamma_t).
\]
The fixed point \(0\) removes the phase freedom in the
constant-speed parametrization.

The following closed-curve estimate is the geometric input behind
the anchored profile argument.  We include its proof in order to make
clear that neither length preservation nor the particular elastic
flow studied in \cite{RuppSpener} is being used.

\begin{lemma}[Anchored constant-speed estimate for normal families]
\label{lem:anchored-gauge}
Let \(I\subset\R\) be an interval and let
\(f\colon I\times\Sph^1\to\R^d\) be a smooth family of immersed
closed curves whose velocity is everywhere normal, namely
\begin{equation}
 \label{eq:normal-curve-family}
 \langle\partial_tf,\partial_xf\rangle=0.
\end{equation}
Let \(\widetilde f\) be the anchored constant-speed parametrization
defined as in \eqref{eq:anchored-gauge}.  Write
\[
 L(f)=\int_{\Sph^1}\dd s_f,\qquad
 E_{\R^d}(f)=\frac12\int_{\Sph^1}|\vec\kappa_f|^2\dd s_f .
\]
Then, at every time,
\begin{equation}
 \label{eq:closed-gauge-estimate}
 \|\partial_t\widetilde f\|_{L^2(\dd x)}
 \leq
 \left(\frac2{L(f)}+16E_{\R^d}(f)\right)^{1/2}
 \|\partial_tf\|_{L^2(\dd s_f)} .
\end{equation}
\end{lemma}

\begin{proof}
Fix a time and abbreviate \(V:=\partial_tf\), \(L:=L(f)\), and
\(q:=\langle V,\vec\kappa_f\rangle\).  If \(s=s(x)\) denotes
arc length from the anchored point, the normality condition
\eqref{eq:normal-curve-family} gives
\[
 \partial_t\dd s_f=-q\,\dd s_f.
\]
Consequently, differentiation of
\[
 \phi(x)=L^{-1}\int_0^x|\partial_zf|\dd z,\qquad
 \psi=\phi^{-1},
\]
and the first variation formula for the length of a closed curve give
\[
 L'=-\int_0^Lq\,\dd s
\]
and
\[
 |(\partial_xf)\,\partial_t\psi|\circ\phi
 =
 \left|
   -\frac{s}{L}\int_0^Lq\,\dd s+\int_0^sq\,\dd r
 \right| .
\]
Since \(\dd(\phi(x))=\dd s/L\), a change of variables yields
\[
 \int_{\Sph^1}|V\circ\psi|^2\dd x
 =\frac1L\|V\|_{L^2(\dd s_f)}^2.
\]
Moreover,
\begin{align*}
 \int_{\Sph^1}|(\partial_xf\circ\psi)\partial_t\psi|^2\dd x
 &\leq
 4\left(\int_{\Sph^1}|q|\dd s_f\right)^2\\
 &\leq
 4\|V\|_{L^2(\dd s_f)}^2
   \|\vec\kappa_f\|_{L^2(\dd s_f)}^2\\
 &=8E_{\R^d}(f)\|V\|_{L^2(\dd s_f)}^2 .
\end{align*}
Finally,
\(\partial_t\widetilde f=V\circ\psi+
 (\partial_xf\circ\psi)\partial_t\psi\).
The inequality \(|X+Y|^2\leq2|X|^2+2|Y|^2\) proves
\eqref{eq:closed-gauge-estimate}.  This is the closed-curve,
normal-velocity version of
\cite[Lemma~4.10 and Remark~4.11]{RuppSpener}.
\end{proof}

\begin{lemma}[Profile and surface speeds]
\label{lem:profile-surface-speed}
Between a normal-gauge Hopf profile evolution and its corresponding
MIWF trajectory, the following identity holds:
\begin{equation}
 \label{eq:profile-surface-speed}
 \|\partial_t\gamma_t\|_{L^2(\dd s_{\gamma_t})}
 =\frac2{\sqrt\pi}
 \|\partial_tF_t\|_{L^2(\Sigma,\dd\mu_t)} .
\end{equation}
\end{lemma}

\begin{proof}
Put \(J_t:=\nabla_{L^2}E(\gamma_t)\) and
\(G_t:=\nabla_{L^2}\Will(F_t)\).  Propositions~4.3--4.4 of
\cite{Jakob} give, in the normal gauges,
\[
 \partial_t\gamma_t
 =-\frac{J_t}{(1+\kappa_{\gamma_t}^2)^2},
 \qquad
 D\pi_{F_t}(G_t)=4J_t,
 \qquad
 |A_t^0|^2=2(1+\kappa_{\gamma_t}^2).
\]
In the standard normalization used here, the Hopf map multiplies the
lengths of horizontal vectors by $2$:
\begin{equation}
 \label{eq:hopf-horizontal-dilation}
 |D\pi(v)|=2|v|
 \qquad\text{for every horizontal vector }v.
\end{equation}
The normal to a Hopf torus is horizontal, so the middle identity above
and \eqref{eq:hopf-horizontal-dilation} give \(|G_t|=2|J_t|\).
Consequently the normal speeds agree pointwise along a horizontal
lift up to the constant factor $2$:
\[
 |\partial_tF_t|
 =|A_t^0|^{-4}|G_t|
 =\frac{|J_t|}{2(1+\kappa_{\gamma_t}^2)^2}
 = \frac12\, |\partial_t\gamma_t|.
\]
Reparametrization changes only the tangential component of velocity,
so the normal MIWF speed of \(F_t\) is the one computed above.
Finally, each Hopf fibre has length
\(2\pi\), whereas a horizontal lift of a profile tangent has half its
length by \eqref{eq:hopf-horizontal-dilation}.  Thus the fibre-integral
factor in the area element is \(\pi\).  By Remark~4.2(2) of
\cite{Jakob}, the bound
\[
 W(F_t)\leq W_0\leq8\pi<4\pi^2
\]
forces every parametrization \(F_t\) to be simple, so its profile
traverses its trace once.  Accordingly Proposition~4.3 of
\cite{Jakob} gives
\[
 \int_\Sigma h\,\dd\mu_t
 =\pi\int_{\Sph^1}h\,\dd s_{\gamma_t}
\]
for every function \(h\) pulled back from the profile.  Taking
\(h=\frac14|\partial_t\gamma_t|^2\) proves
\eqref{eq:profile-surface-speed}.
\end{proof}

\begin{theorem}[Full convergence of a fixed anchored finite-time Hopf profile]
\label{thm:anchored-profile}
Under the hypotheses of Corollary~\ref{cor:hopf}, fix the oriented
normal gauge and the material anchor introduced above, and let
\(\widetilde\gamma_t\) be the anchored constant-speed profile
\eqref{eq:anchored-gauge}.  Set
\begin{equation}
 \label{eq:gauge-constant}
 C_{\mathrm g}:=\left(\frac2\pi+64\right)^{1/2}.
\end{equation}
Then, for \(0\leq s<t<T\),
\begin{equation}
 \label{eq:anchored-profile-cauchy}
 \|\widetilde\gamma_t-\widetilde\gamma_s\|_{L^2(\Sph^1)}
 \leq
 \frac{C_{\mathrm g}}{\sqrt\pi}
 \sqrt{(t-s)\bigl(W(s)-W(t)\bigr)} .
\end{equation}
There are an embedded regular constant-speed curve
\(\gamma_T\in W^{2,2}(\Sph^1,\Sph^2)\cap C^1(\Sph^1,\Sph^2)\)
and a number \(L_T\in[\pi,8]\) such that, as \(t\nearrow T\),
\begin{align}
 \widetilde\gamma_t&\rightharpoonup\gamma_T
 &&\text{weakly in }W^{2,2}(\Sph^1,\R^3),\label{eq:profile-weak}\\
 \widetilde\gamma_t&\longrightarrow\gamma_T
 &&\text{strongly in }W^{1,2}(\Sph^1,\R^3),\label{eq:profile-w12}\\
 \widetilde\gamma_t&\longrightarrow\gamma_T
 &&\text{strongly in }C^{1,\alpha}(\Sph^1,\R^3),
   \quad 0<\alpha<\tfrac12,\label{eq:profile-strong}\\
 L(\widetilde\gamma_t)&\longrightarrow L_T=L(\gamma_T),
 \qquad |\partial_x\gamma_T|=L_T .\label{eq:length-limit}
\end{align}
The trace of \(\gamma_T\) is the embedded \(C^1\) profile
\(\pi(\spt\Var_T)\) of the unique limiting Hopf torus.
\end{theorem}

\begin{proof}
The simplicity criterion in Remark~4.2(2) of \cite{Jakob} applies
because
\[
 W(F_0)\leq8\pi<4\pi^2.
\]
Thus \(F_0\) is simple.  The energy identity in Proposition~4.3 of
\cite{Jakob} then gives
\[
 E(\gamma_0)=\frac{W(F_0)}{\pi}\leq8.
\]
Proposition~4.6 of \cite{Jakob} gives, for \(0\leq t<T\),
\begin{equation}
 \label{eq:hopf-energy-length-bounds}
 \pi\leq L(\gamma_t)\leq E(\gamma_t)\leq E(\gamma_0)\leq8 .
\end{equation}
For a curve in the unit sphere,
\[
 \vec\kappa_{\R^3}=-\gamma+\kappa_\gamma\nu_\gamma,
 \qquad
 E_{\R^3}(\gamma)=\frac12E(\gamma).
\]
Hence Lemma~\ref{lem:anchored-gauge},
\eqref{eq:hopf-energy-length-bounds}, and
Lemma~\ref{lem:profile-surface-speed} imply
\begin{equation}
 \label{eq:anchored-speed-bound}
 \|\partial_t\widetilde\gamma_t\|_{L^2(\dd x)}
 \leq\frac{2C_{\mathrm g}}{\sqrt\pi}
 \|\partial_tF_t\|_{L^2(\dd\mu_t)} .
\end{equation}
On the other hand, \eqref{eq:hopf-nonumbilic} and the dissipation
identity \eqref{eq:dissipation} give
\[
 \|\partial_tF_t\|_{L^2(\dd\mu_t)}^2
 =\int_\Sigma|A_t^0|^{-8}|G_t|^2\dd\mu_t
 \leq\frac14\int_\Sigma|A_t^0|^{-4}|G_t|^2\dd\mu_t
 =-\frac14W'(t).
\]
Integrating \eqref{eq:anchored-speed-bound} on \([s,t]\) and applying
Cauchy--Schwarz proves \eqref{eq:anchored-profile-cauchy}.  Since
\(T<\infty\) and \(W\) is decreasing and bounded below, the right-hand
side tends to zero as \(s,t\nearrow T\).  Thus
\(\widetilde\gamma_t\) is Cauchy in \(L^2\), with a unique limit
\(\gamma_T\).

It remains to upgrade the convergence.  By
\eqref{eq:constant-profile-speed} and the sphere Frenet formula,
\[
 \partial_x^2\widetilde\gamma_t
 =L(\widetilde\gamma_t)^2
   \bigl(-\widetilde\gamma_t+
       \kappa_{\widetilde\gamma_t}
       \nu_{\widetilde\gamma_t}\bigr),
\]
and therefore
\begin{equation}
 \label{eq:profile-w22-bound}
 \|\partial_x^2\widetilde\gamma_t\|_{L^2(\dd x)}^2
 =L(\widetilde\gamma_t)^3E(\widetilde\gamma_t)\leq8^4 .
\end{equation}
Together with \eqref{eq:constant-profile-speed}, this is a uniform
\(W^{2,2}\) bound.  Every sequence
\(\{\widetilde\gamma_{t_j}\}\) with \(t_j\nearrow T\) consequently
has a subsequence converging weakly in \(W^{2,2}\) and strongly in
\(W^{1,2}\) and \(C^{1,\alpha}\), for every
\(0<\alpha<\frac12\).  Its \(L^2\) limit must be the already fixed
\(\gamma_T\).  A contradiction argument with an arbitrary sequence
therefore proves the full convergences
\eqref{eq:profile-weak}--\eqref{eq:profile-strong}.

Uniform convergence of the derivatives and
\eqref{eq:constant-profile-speed} show that
\(L(\widetilde\gamma_t)\) converges to a number \(L_T\) and that
\(|\partial_x\gamma_T|=L_T\).  The bounds
\eqref{eq:hopf-energy-length-bounds} give \(L_T\in[\pi,8]\), so the
limit is regular and \(L_T=L(\gamma_T)\).

Finally, for every \(0\leq t<T\),
\[
 F_t(\Sigma)=\pi^{-1}\bigl(\gamma_t(\Sph^1)\bigr).
\]
The strong \(C^1\) profile convergence gives Hausdorff convergence of
the projected traces to \(\gamma_T(\Sph^1)\).  Corollary~\ref{cor:hopf}
and continuity of the Hopf map give convergence of the same traces to
\(\pi(\spt\Var_T)\).  Hence
\[
 \gamma_T(\Sph^1)=\pi(\spt\Var_T).
\]
The right-hand side is the embedded \(C^1\) profile supplied by
\cite[Theorem~1.2(1)]{Jakob}.  It remains to verify that
\(\gamma_T\) traverses this profile exactly once.  By
Proposition~4.6(2) of \cite{Jakob}, \(\widetilde\gamma_t\) is an
embedding for every \(t>0\) under the energy bound
\eqref{eq:hopf-energy-length-bounds}.  Choose \(t\) sufficiently close
to \(T\), and put
\(C:=\gamma_T(\Sph^1)\).  Since \(\gamma_T\) is regular and \(C\) is
an embedded circle, \(\gamma_T\colon\Sph^1\to C\) is a finite covering
of some positive degree \(m\).  Choose an annular neighborhood \(U\)
of the embedded \(C^1\) circle \(C\) in \(\Sph^2\), together with a
deformation retraction \(P\colon U\to C\).  The strong \(C^1\)
convergence implies that, for all \(t\) sufficiently close to \(T\),
the embedded image
\(\widetilde\gamma_t(\Sph^1)\) lies in \(U\) and
\(P\circ\widetilde\gamma_t\) is homotopic in \(C\) to
\(P\circ\gamma_T=\gamma_T\).  Thus
\[
 \deg(P\circ\widetilde\gamma_t)=m.
\]
This degree is nonzero, so \(\widetilde\gamma_t(\Sph^1)\) is an
essential embedded circle in the annulus \(U\).  The homotopy class of
an essential embedded circle in an annulus is primitive in
\(\pi_1(U)\simeq\mathbb Z\); hence its degree is \(1\) or \(-1\).
The consistently chosen orientation makes it positive.  Therefore
\(m=1\), and \(\gamma_T\) is an embedding which traverses \(C\)
exactly once.  This completes the identification.
\end{proof}

\begin{remark}[No unproved strong-curvature conclusion]
\label{rem:no-strong-profile}
Theorem~\ref{thm:anchored-profile} does not assert strong
\(W^{2,2}\) convergence.  The estimate determines the full first-order
profile limit, while \eqref{eq:profile-w22-bound} gives only weak
compactness of the curvature.  Strong \(W^{2,2}\) convergence would
require an additional no-defect or endpoint energy identity.
\end{remark}

\begin{corollary}[Unique Pinkall-modulus limit]
\label{cor:pinkall-modulus}
Under the hypotheses of Theorem~\ref{thm:anchored-profile}, orient the
profiles consistently with the reduced flow.  Let
\(\mathcal A_t\in\R/4\pi\mathbb Z\) be the oriented area class
enclosed by \(\gamma_t\), and let \(\mathcal A_T\) be the corresponding 
area class for \(\gamma_T\).  Then
\begin{equation}
 \label{eq:area-class-limit}
 \mathcal A_t\longrightarrow\mathcal A_T
 \qquad\text{in }\R/4\pi\mathbb Z .
\end{equation}
For \(t\) sufficiently close to \(T\), choose compatible real
representatives such that \(\mathcal A_t\to\mathcal A_T\).  Pinkall's
lattice is
\begin{equation}
 \label{eq:pinkall-lattice}
 \Gamma_t=
 \operatorname{span}_{\mathbb Z}
 \left\{(2\pi,0),
 \left(\frac{\mathcal A_t}{2},\frac{L(\gamma_t)}2\right)\right\},
\end{equation}
and the corresponding modulus is
\begin{equation}
 \label{eq:pinkall-modulus}
 \tau_t=\frac{\mathcal A_t+iL(\gamma_t)}{4\pi}\in\mathbb H .
\end{equation}
It satisfies
\[
 \tau_t\longrightarrow
 \tau_T:=\frac{\mathcal A_T+iL_T}{4\pi}.
\]
Its imaginary part satisfies the quantitative bounds
\begin{equation}
 \label{eq:modulus-imaginary-bounds}
 \frac14\leq\operatorname{Im}\tau_T\leq\frac2\pi.
\end{equation}
Consequently the conformal classes
\[
 [\tau_t]\longrightarrow[\tau_T]
 \qquad\text{in }\mathcal M_1=\mathbb H/\mathrm{PSL}_2(\mathbb Z)
\]
have a unique finite-time limit, independent of the approaching time
sequence.  With the orientation induced by the fixed oriented profile
and the positively oriented Hopf fibres, the class \([\tau_T]\) is the
intrinsic oriented conformal class of the limiting \(C^1\) Hopf torus.
In particular, no conformal
degeneration occurs at the finite endpoint.
\end{corollary}

\begin{proof}
Pinkall equips the Hopf bundle with the connection whose horizontal
spaces are orthogonal to the fibres.  Its curvature is
\(\frac12\dd V_{\Sph^2}\), and the holonomy angle \(\delta_\gamma\)
of an oriented closed profile satisfies
\[
 \delta_\gamma=\frac{\mathcal A_\gamma}{2}\pmod{2\pi}.
\]
Parallel transport for a fixed smooth connection depends continuously
on the path in the \(C^1\) topology.  Therefore
\eqref{eq:profile-strong} implies
\[
 \exp\!\left(\frac{i\mathcal A_t}{2}\right)
 \longrightarrow
 \exp\!\left(\frac{i\mathcal A_T}{2}\right),
\]
which is precisely \eqref{eq:area-class-limit}.  A local lift of the
covering \(\R\to\R/4\pi\mathbb Z\) supplies compatible representatives
with \(\mathcal A_t\to\mathcal A_T\) on a terminal time interval.

Proposition~1 of \cite{Pinkall} states that the flat Hopf torus over a
closed profile of length \(L\) and oriented area \(\mathcal A\) is
\(\R^2/\Gamma\), where \(\Gamma\) is generated by
\((2\pi,0)\) and \((\mathcal A/2,L/2)\).  Identifying \(\R^2\) with \(\mathbb C\)
and dividing the second generator by the first gives
\(\tau=(\mathcal A+iL)/(4\pi)\).  Changing the representative of
\(\mathcal A\) by
\(4\pi\) replaces \(\tau\) by \(\tau+1\) and therefore does not change
its class in \(\mathcal M_1\).  The convergence of the area classes,
together with \eqref{eq:length-limit}, proves convergence of the
moduli.  The bounds \eqref{eq:modulus-imaginary-bounds} follow at once
from $L_T\in[\pi,8]$.

For completeness, let \(\eta_T\) be a \(C^1\) horizontal lift of the
regular constant-speed profile \(\gamma_T\), parametrized by horizontal
arc length \(y\in\R\).  In fibre-angle and horizontal arc-length
coordinates, the map
\[
 X_T(y,\theta):=e^{i\theta}\eta_T(y)
\]
is a \(C^1\) local isometry from the Euclidean plane onto the limiting
Hopf torus.  Its deck transformations are determined by the fibre
period \(2\pi\) and the terminal holonomy
\(\mathcal A_T/2\pmod{2\pi}\).  With the connection and orientation
conventions fixed above, their lattice is therefore exactly
\[
 \Gamma_T=\operatorname{span}_{\mathbb Z}
 \left\{(2\pi,0),
 \left(\frac{\mathcal A_T}{2},\frac{L_T}{2}\right)\right\}.
\] 
Thus Pinkall's covering calculation extends directly to the regular
\(C^1\) profile and identifies \([\tau_T]\) with the intrinsic oriented
conformal class of the induced metric on the limiting Hopf torus; 
compare here also with Section 3 of \cite{Ahmedou.JakobHopf}.

\end{proof}

\begin{remark}[No restart conclusion in the Hopf sector]
The identity \eqref{eq:hopf-nonumbilic} holds along the smooth
trajectory.  The anchored-profile and modulus limits do not assert
that the \(C^1\) limiting support possesses a pointwise defined,
quantitatively nonvanishing trace-free second fundamental form, nor
that the MIWF can be restarted from the limiting parametrization.  The
continuation issues listed in Section~5 remain logically separate.
\end{remark}

\end{document}